\documentclass[oneside]{amsart}
\usepackage{amsthm}
\usepackage{amssymb}
\usepackage{graphicx}
\numberwithin{equation}{section} 

\DeclareFontFamily{OT1}{rsfs}{}
\DeclareFontShape{OT1}{rsfs}{m}{n}{ <-7> rsfs5 <7-10> rsfs7 <10-> rsfs10}{}
\DeclareMathAlphabet{\mycal}{OT1}{rsfs}{m}{n}

\theoremstyle{plain}
\newtheorem{thm}{Theorem}
  \newtheorem{defn}[thm]{Definition}
  \newtheorem{lem}[thm]{Lemma}
  
  \newtheorem{prop}[thm]{Proposition}
\theoremstyle{remark}
  \newtheorem{rem}[thm]{Remark}
\theoremstyle{plain}
\newtheorem{assume}{Assumption}

\newcommand{\uc}[1]{\mathcal{U}\left(#1\right)}
\begin{document}

\advance\textwidth 3cm
\title[On uniqueness for anisotropic Maxwell's equations]{On uniqueness for time harmonic anisotropic Maxwell's equations with
piecewise regular coefficients}

\author{John M. Ball}
\author{Yves Capdeboscq}
\author{Basang Tsering Xiao}
\address{Mathematical Institute, 24-29 St Giles, OXFORD OX1 3LB, UK}

\begin{abstract}
We are interested in the uniqueness of solutions to Maxwell's
equations when the magnetic permeability $\mu$ and the permittivity
$\varepsilon$ are symmetric positive definite matrix-valued
functions in $\mathbb{R}^{3}$. We show that a unique continuation
result for globally $W^{1,\infty}$ coefficients in a smooth, bounded
domain, allows one to prove that the solution is unique in the case
of coefficients which are piecewise $W^{1,\infty}$ with respect to a suitable
countable collection of sub-domains with $C^{0}$ boundaries. Such suitable
collections include any bounded finite collection.
The proof relies on a general argument, not specific to Maxwell's equations. This
result is then extended to the case when within these sub-domains the permeability
and permittivity are only $L^\infty$ in sets of small measure.
\end{abstract}
\maketitle

\section{Introduction}

Suppose we are given a time-harmonic incident electric field
$\mycal{E}^i$ and magnetic field $\mycal{H}^i$,  special solutions
of the time-harmonic homogeneous linear Maxwell equations of the
form $\mycal{E}^i=\Re\left(\mathbf{E}^i e^{-i \omega t}\right)$ and
magnetic field $\mycal{H}^i= \Re\left( \mathbf{H}^i e^{-i\omega t}
\right)$, where $\mathbf{E}^{i} \in
H_{\textrm{loc}}^{1}\left(\mathbb{R}^{3}\right)^{3}$ and
$\mathbf{H}^{i}\in
H_{\textrm{loc}}^{1}\left(\mathbb{R}^{3}\right)^{3}$ are complex-valued solutions of
the homogeneous time-harmonic Maxwell equations
\begin{align*}
\nabla\wedge\mathbf{E}^{i}-i\,\omega\mu_{0}\mathbf{H}^{i}
                            & =0\textrm{ in }\mathbb{R}^{3},\\
\nabla\wedge\mathbf{H}^{i}+i\,\omega\varepsilon_{0}\mathbf{E}^{i}
                            & =0\textrm{ in }\mathbb{R}^{3},
\end{align*}
where $\mu_{0}$ and $\varepsilon_{0}$ are positive constants,
representing respectively the magnetic permeability and the
electric permittivity of vacuum, and $\omega\in\mathbb{R}\setminus\{0\}$.
The full time-harmonic electromagnetic field $(\mathbf{E},\mathbf{H})\in H_{\textrm{loc}}\left(\textrm{curl};\mathbb{R}^{3}\right)$,
where for any domain $W$ we define
$$
H_{\textrm{loc}}\left(\textrm{curl};W \right):=
\left\{
\mathbf{u}\in L_{\textrm{loc}}^{2}\left(W\right)^{3}
\textrm{ such that }
\nabla\wedge\mathbf{u}\in L_{\textrm{loc}}^{2}\left(W\right)^{3}
\right\},
$$
satisfies Maxwell's equations
\begin{align}
\nabla\wedge\mathbf{E}  -i\,\omega\mu_{0}\mu(x)\,\mathbf{H}
                        & = 0\textrm{ in }\mathbb{R}^{3},\label{eq:maxwell-system}\\
\nabla\wedge\mathbf{H}  +i\,\omega\varepsilon_{0}\varepsilon(x)\,\mathbf{E}
                        & = 0\textrm{ in }\mathbb{R}^{3},\nonumber
\end{align}
where $\varepsilon$ and $\mu$ are real matrix-valued functions in
$L^{\infty}\left(\mathbb{R}^{3}\right)^{3\times3}$.
Decomposing the full electromagnetic field into its incident part and its scattered
part,
\begin{equation}\label{eq:def-scat}
\mathbf{E}^{s}:=\mathbf{E}-\mathbf{E}^{i}, \textrm{ and }
\mathbf{H}^{s}:=\mathbf{H}-\mathbf{H}^{i},
\end{equation}
we assume that the scattered field satisfies the Silver-M\"{u}ller
radiation condition, uniformly in all directions, 
that is, if $x:=r\theta$ then
\begin{equation}\label{eq:radiation-condition}
\lim_{r\to\infty}\sup_{\theta\in S^{2}}
\left|
\mathbf{H}^{s}\left(r \theta\right)\wedge r\theta
-r\mathbf{\mathbf{E}}^{s}\left(r\theta\right)
\right|=0,
\end{equation}
where $S^2:=\{ x \in \mathbb{R}^3 \textrm{ such that } |x|=1\}$ denotes the unit sphere.

This paper is about the existence of a unique solution to
\eqref{eq:maxwell-system} satisfying \eqref{eq:def-scat} and
\eqref{eq:radiation-condition}, under the following additional
hypotheses on $\varepsilon$ and $\mu$. We assume that both
permittivity and permeability are real symmetric, uniformly positive
definite and bounded, that is, there
exist $0<\alpha\leq\beta<\infty$ such that for all $\xi\in\mathbb{R}^{3}$  and almost every $x\in\mathbb{R}^3$,
\begin{align}
\alpha\left|\xi\right|^{2}&\leq\varepsilon(x)\xi\cdot\xi\leq\beta \left|\xi\right|^{2},\label{eq:eps-coercive1}\\
\alpha\left|\xi\right|^{2}&\leq\mu(x)\xi\cdot\xi\leq\beta\left|\xi\right|^{2}.\label{eq:eps-coercive2}
\end{align}
We suppose that $\varepsilon$ and $\mu$ vary only in an open bounded domain $\Omega$, so that
\begin{equation}\label{eq:eps-local}
\varepsilon=\mu=\mathbf{I}_3\textrm{ in } {\Omega}^{c}=\mathbb{R}^{3}\setminus {\Omega},
\end{equation}
where $\mathbf{I}_3$ is the identity matrix in $\mathbb{R}^{3\times3}$.
We assume that $\Omega$ is of the form
\begin{equation}
\Omega=\textrm{int}\left({\underset{i\in
I}{\cup}}\bar{\Omega}_{i}\right),\label{eq:def-Omea}
\end{equation}
where the sub-domains $\Omega_{i}$, $i\in I\subset\mathbb{N}$ are disjoint and of class $C^{0}$, and int 
denotes the interior. The permittivity $\varepsilon$ and the permeability $\mu$ are assumed to be 
piecewise $W^{1,\infty}$ with respect to the sub-domains $\Omega_i$, so that for each $i\in I$, there exist
 $\varepsilon_{i},\mu_{i} \in W^{1,\infty}\left(\mathbb{R}^{3}\right)^{3\times3}$ satisfying 
\eqref{eq:eps-coercive1}-\eqref{eq:eps-coercive2} and 
\begin{equation}\label{eq:lip}
\| \varepsilon_{i} \|_{W^{1,\infty}(\mathbb{R}^3)^{3\times3}} + \| \mu_{i} \|_{W^{1,\infty}(\mathbb{R}^3)^{3\times3}} \leq M_i,
\end{equation}
where $M_i>0$ is a positive constant, such that
\begin{equation} \label{eq:localreg-eps}
\varepsilon(x)=\varepsilon_{i}(x) \textrm{ and } \mu(x)=\mu_{i}(x),
\quad \textrm{a.e. } x \in \Omega_i \,.
\end{equation}
Given a bounded set $A\subset\mathbb{R}^3$, we write $\uc{A}$ as the (unique) unbounded component of $\bar{A}^c$.
\begin{assume}\label{Assumption}
For any $J\subset  I$, and $\Omega_J=  \textrm{\rm int}\left({\underset{j\in J}{\cup}}\bar{\Omega}_{j}\right)$,
there exists $j_0\in J$ such that $\partial\uc{\Omega_J}\cap\partial\Omega_{j_0}$ admits an interior point
relative to $\partial\uc{\Omega_J}$. In other words, there exist $j_0\in J$ and $x_0\in \partial\uc{\Omega_J}\cap\partial\Omega_{j_0}$
such that $B(x_0,\delta)\cap \partial\uc{\Omega_J} \subset \partial\Omega_{j_0}$ for some $\delta>0$.
\end{assume}

\begin{prop}
Assumption~\ref{Assumption} holds if for all $J\subset I$, 
there exist $x_J\in\partial\uc{\Omega_J}$ and $\delta_J>0$ such that $ B(x_J,\delta_J) \cap \Omega_j\neq\emptyset$ 
for only finitely many $j\in J$. In particular, Assumption~\ref{Assumption} holds when $I$ is finite.
\end{prop}
\begin{proof}
Given $J$, $x_J\in\partial\uc{\Omega_J}$ and  $\delta_J$ as in the statement of the proposition 
let  $B_J = B(x_J,\delta_J)$ and let $J^\prime$ be the finite subset of $J$ such that
$ B_J \cap \underset{j\in J}{\cup}\Omega_j= B_J \cap \underset{j\in J^\prime}{\cup}\Omega_j$. We first show that
\begin{align}
\partial\uc{\Omega_J}\cap B_J = 
\underset{j\in J}{\cup}  \left(\partial\uc{\Omega_J}\cap  B_J\right) \cap \partial{\Omega}_j.\label{eq:hypobaire}
\end{align}
Indeed, let $x\in\partial\uc{\Omega_J}\cap B_J$. Then $x\not\in\underset{j\in J}{\cup} \Omega_j$. 
We claim that there exists a sequence $x_k\in {\underset{j\in J}{\cup}}{\Omega}_{j}$ such that $x_k$ tends to $x$.
If not, for some $\eta>0$ sufficiently small, we would have $B(x,\eta) \subset B_J$ , 
and $B(x,\eta)\cap {\underset{j\in J}{\cup}}{\Omega}_{j} = B(x,\eta)\cap {\underset{j\in J^\prime}{\cup}}{\Omega}_{j} 
=\emptyset$. 
On the other hand, there exists a sequence $y_k\in \uc{\Omega_J}$ such that $y_k$ tends to $x$.
But $B(x,\eta)$ is connected and contained in $\bar{\Omega}_{J}^c$,
thus $B(x,\eta)\subset \uc{\Omega_J}$. This contradiction proves the claim. 

Next, we note that $\partial\uc{\Omega_J}$ is closed, thus complete in the subspace topology induced by $\mathbb{R}^3$. 
Its intersection with the open ball $B_J$ is an open subspace of $\partial\uc{\Omega_J}$ by definition of the subspace 
topology. It is therefore a Baire space (see e.g. \cite{MUNKRES-75}). If a Baire Space is a countable union of closed sets, 
then one of the sets has an interior point. Using the identity \eqref{eq:hypobaire}, we obtain that
there exists $j_0$ such that  $\partial \uc{\Omega_J} \cap   B_J \cap \partial \Omega_{j_0}$ admits an interior point relative to
$\partial \uc{\Omega_J} \cap  B_J$, that is, there exist $j_0 \in J$, $x_0\in \partial \uc{\Omega_J} \cap  B_J$ and $\delta>0$ such that
$B(x_0,\delta )\cap \partial\uc{\Omega_J}\cap B_J \subset \partial\Omega_{j_0}$. 

Since $B_J$ is open, $B(x_0,\delta )\cap B_J = B(x_0,\delta )$ when $\delta$ is sufficiently small, and we have established that Assumption~\ref{Assumption} holds.
\end{proof}

\begin{figure}\label{fig:examples}
\includegraphics[width=0.3\columnwidth]{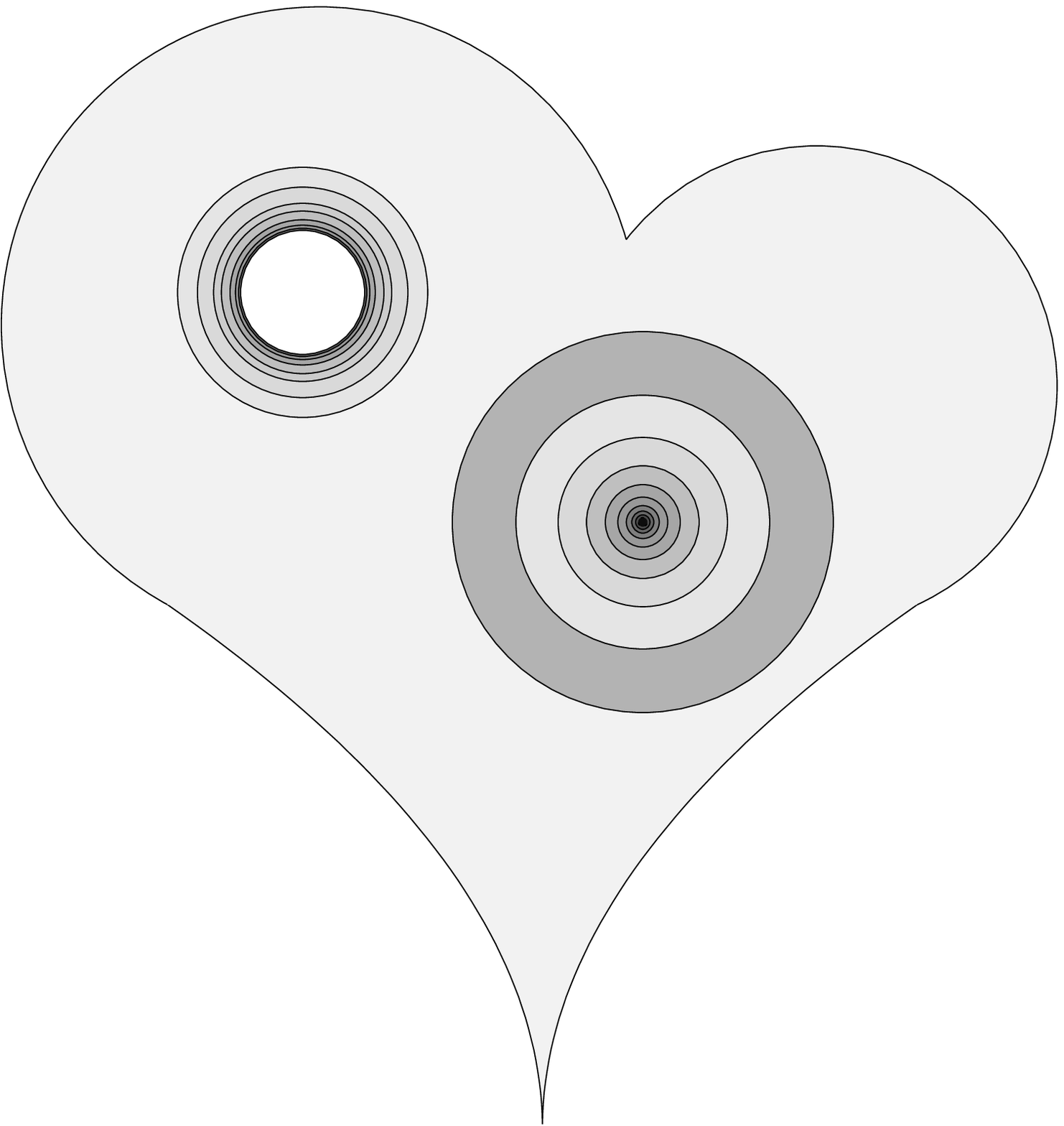}\, \includegraphics[width=0.3\columnwidth]{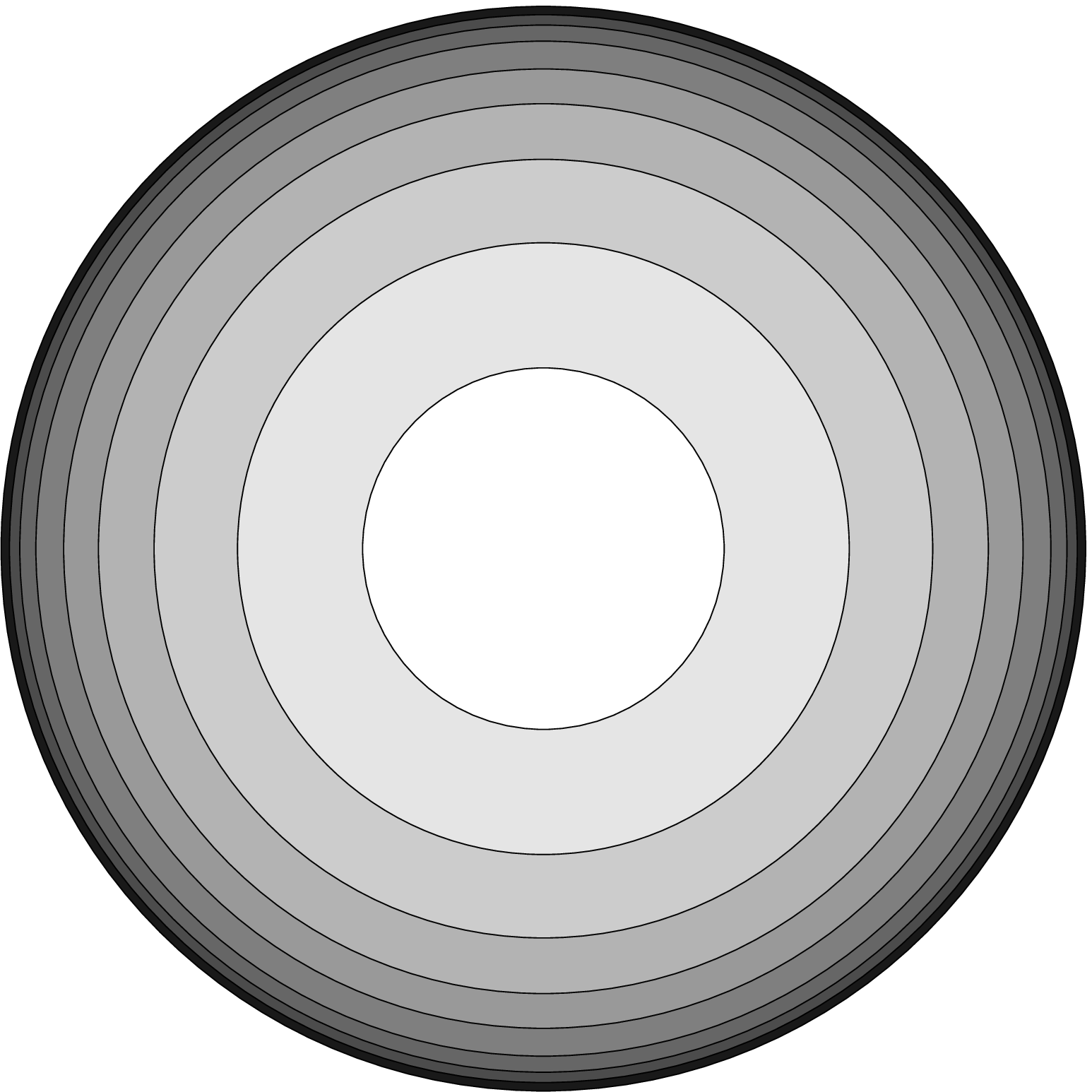}\, \includegraphics[width=0.3\columnwidth]{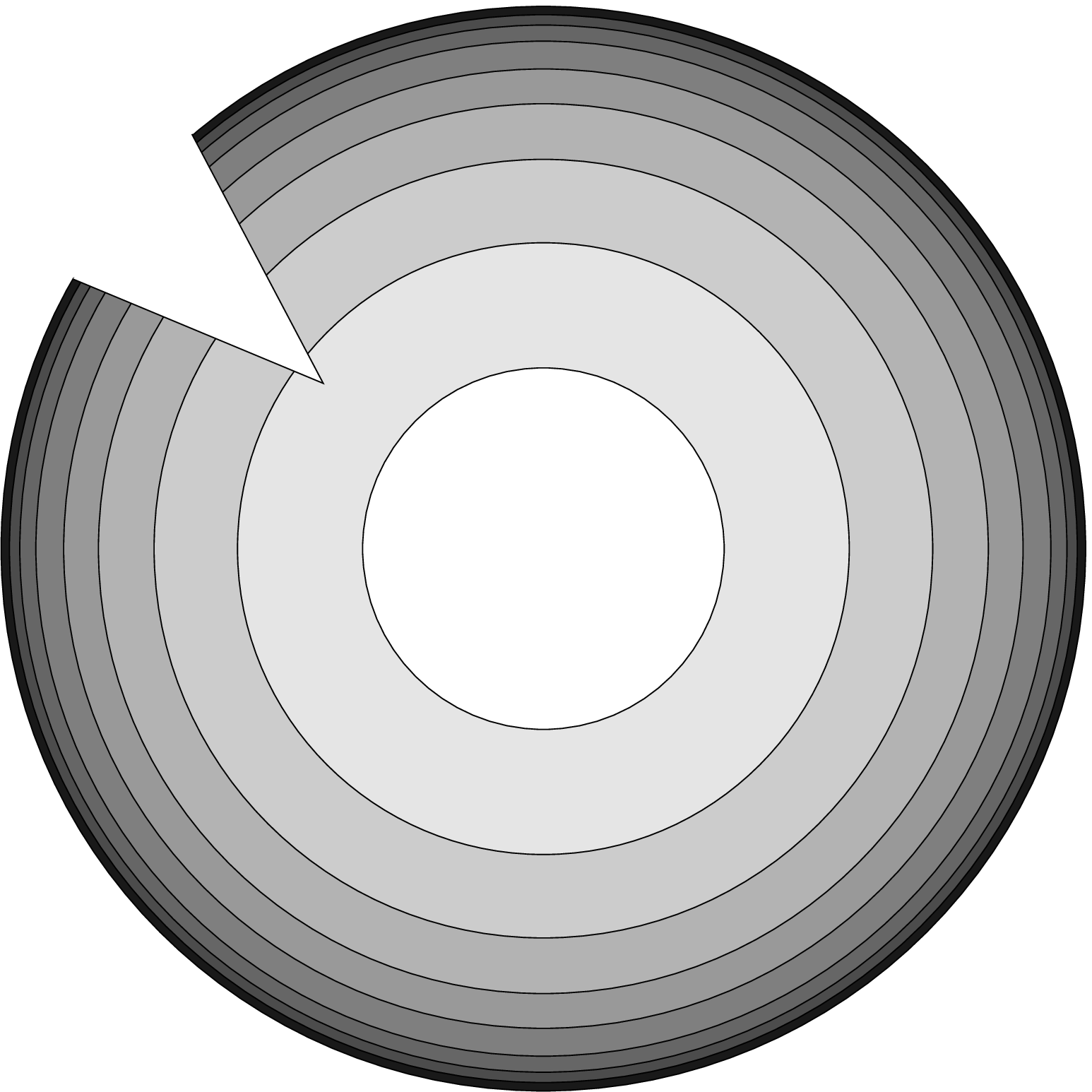}
\caption{
Left: an infinite collection of sub-domains satisfying Assumption~\ref{Assumption}. Centre: an infinite collection of sub-domains excluded by
Assumption~\ref{Assumption}. Right: this collection satisfies Assumption~\ref{Assumption}.}
\end{figure}

An example of a collection of sub-domains excluded by Assumption~\ref{Assumption} is a collection of  concentric shells concentrating on an exterior boundary, such as
\begin{equation}\label{eq:contrex}
\Omega_{i} = B\left(\mathbf{0},\frac{i}{i+1}\right)\setminus \bar{B}\left(\mathbf{0},\frac{i-1}{i}\right), \quad i=1,2,3,\ldots
\end{equation}
In such a case, $\partial \uc{\Omega}$ is the unit sphere, which is not the boundary of any of the subsets. On the other hand, Assumption~\ref{Assumption} allows
the sub-domains $\Omega_i$  to concentrate at a point or near an interior boundary. In Figure~\ref{fig:examples}, we represent on the left
a non-Lipschitz non-simply connected domain $\Omega$  which satisfies Assumption~\ref{Assumption}. In the centre, the domain given by \eqref{eq:contrex}
excluded by Assumption~\ref{Assumption} is shown. On the right, we sketch a domain inspired by the one described by \eqref{eq:contrex} which satisfies
Assumption~\ref{Assumption}: near the accumulating boundary, interior points can be found on the wedge-shaped slit in the domain.
\section{Main result}
Our main result is the following theorem.
\begin{thm} \label{thm:no-defect}
Assume that \eqref{eq:eps-coercive1}--\eqref{eq:localreg-eps}  and Assumption~\ref{Assumption} hold.
If for a given $\omega\neq0$,
$\mathbf{E}\in H_{\rm{loc}}\left(\rm{curl};\mathbb{R}^{3}\right)$
and
$\mathbf{H}\in H_{\rm{loc}}\left(\rm{curl};\mathbb{R}^{3}\right)$
are solutions of
\eqref{eq:maxwell-system}--\eqref{eq:radiation-condition}
corresponding to $\mathbf{E}^{i}=0$ and $\mathbf{H}^{i}=0$,
then
$\mathbf{E}=\mathbf{H}=0$.
\end{thm}
There is a very long history concerning this problem, under various
assumptions on the coefficients, see e.g.
\cite{CAKONI-COLTON-03,COLTON-KRESS-98,ELLER-YAMAMOTO-06,HAZARD-LENOIR-96,LEIS-68,MULLER-69,NGUYEN-WANG-11,NEDELEC-01,OKAJI-02,VOGELSANG-91}
and the references therein. The improvement provided by the result in this work is
that we assume that $\varepsilon$, $\mu$ are matrix-valued functions and that
the sub-domains $\Omega_i$ are only of class $C^0$. We do not assume that the sub-domains are Lipschitz as assumed for example in 
\cite{HAZARD-LENOIR-96} for the isotropic (scalar) case.
The authors are not aware of the existence of a general uniqueness result for the above problem when the coefficients are just $C^{0,\alpha}$ H\"older continuous,
with $\alpha<1$. For general elliptic equations, counter-examples to unique continuation, the main technique for proving uniqueness, are known in that case, see \cite{MILLER-74}.
We remind the reader of the definition of a domain of
class $C^{0}$.
\begin{defn}\label{def:co-bdy}
A bounded domain $\Omega$ of $\mathbb{R}^3$ is of class $C^0$ if for any point $x_{0}$
on the boundary $\partial\Omega$, there exists a ball
$B\left(x_{0},\delta\right)$
and an orthogonal coordinate system $(x_1,x_2,x_3)$  with origin at $x_{0}$ such that there exists a continuous
function $f:\, C^{0}\left(\mathbb{R}^{2};\mathbb{R}\right)$ that satisfies
$$
\Omega\cap B\left(x_{0},\delta\right)
=
\left\{ x\in B\left(x_{0},\delta\right) : x_{3}> f\left(x_1, x_2\right)\right\}.
$$
\end{defn}

We define $B_0$ as the smallest open ball containing $\Omega$. Note that the uniqueness of the solution outside $B_0$ is well known,
due to the so-called Rellich's Lemma, see e.g. \cite{COLTON-KRESS-98}.
\begin{lem}[Rellich's Lemma]\label{lem:rellich}
If for a fixed $\omega$, $\mathbf{E}\in H_{{\rm{loc}}}\left({\rm{curl}};\mathbb{R}^{3}\right)$
and
$\mathbf{H}\in H_{{\rm{loc}}}\left({\rm{curl}};\mathbb{R}^{3}\right)$
are solutions of \eqref{eq:maxwell-system}-\eqref{eq:radiation-condition}
corresponding to
$\mathbf{E}^{i}=0$ and $\mathbf{H}^{i}=0$,
then
$\mathbf{E}=\mathbf{H}=0$ in $\bar{B}_0^{c}$.
\end{lem}
Our proof relies on a recent unique continuation result \cite{NGUYEN-WANG-11}
proved for globally $W^{1,\infty}$ regular coefficients.

\begin{thm}[\cite{NGUYEN-WANG-11}]\label{thm:unicont}
Let $V$ be a connected open set in $\mathbb{R}^3$.  Assume that $\varepsilon$ and $\mu$ are two real symmetric matrix valued functions in $V$ satisfying 
\eqref{eq:eps-coercive1}--\eqref{eq:eps-coercive2}, 
and 
$$
\| \varepsilon  \|_{W^{1,\infty}(V)^{3\times3}} + \| \mu \|_{W^{1,\infty}(V)^{3\times3}} \leq M,
$$
for some constant $M>0$.  Suppose $(\mathbf{E},\mathbf{H})\in \left(L^2_{\rm{loc}}\left(V\right)\right)^2$ satisfy
\begin{align*}
\nabla\wedge\mathbf{E}  -i\,\omega\mu_{0}\mu(x)\,\mathbf{H}
                        & = 0\textrm{ in }V,\\
\nabla\wedge\mathbf{H}  +i\,\omega\varepsilon_{0}\varepsilon(x)\,\mathbf{E}
                        & = 0\textrm{ in }V.
\end{align*}
Then, there exist $s>0$ independent of $V$, $\mathbf{E}$ and $\mathbf{H}$, such that if for some $x_0\in V$, and for all $N\in\mathbb{N}$ and all $\delta>0$ sufficiently small,
$$
\int_{B(x_0,\delta)} \left(\left|\mathbf{E}\right|^2+\left|\mathbf{H}\right|^2\right) dx \leq C_N \exp\left(-N \delta^{-s}\right)
$$
for some constant $C_N>0$, then $\mathbf{E}=\mathbf{H}=0$ in $V$.
\end{thm}

The proof of Theorem~\ref{thm:no-defect} consists of three steps. The first two steps are given by the two propositions below.

\begin{prop}\label{prop:step1} Under the hypothesis of Theorem~\ref{thm:no-defect}, suppose that $A\subset\mathbb{R}^3$ is a bounded open set and 
that for almost every $x\in\uc{A}$ either  $\varepsilon(x) = \mu(x) = \mathbf{I}_3$ or $\mathbf{E}(x)=\mathbf{H}(x)=0$.
Then $\mathbf{E}=\mathbf{H}=0$ in $\uc{A}$.
\end{prop}
\begin{proof}
For any $\mathbf{v,w}\in L^2\left(\uc{A}\right)^2$,
we have
\begin{align}
\int_{\uc{A}} \nabla\wedge\mathbf{E}\cdot \mathbf{v}\, dx
-i\omega\mu_{0}\int_{\uc{A}}\mu(x)\mathbf{H}\cdot\mathbf{v}\, dx & =0,
\label{eq:max-step11h}\\
\int_{\uc{A}} \nabla\wedge\mathbf{H} \cdot \mathbf{w}\, dx
+i\omega\varepsilon_{0}\int_{\uc{A}}\epsilon(x)\mathbf{E}\cdot\mathbf{w}\, dx& =0,
\label{eq:max-step1ie}
\end{align}
where the integrals \eqref{eq:max-step11h} and \eqref{eq:max-step1ie} are well defined by Rellich's Lemma~\ref{lem:rellich}.
Since for almost every $x$ in $\uc{A}$, either  $\varepsilon(x)=\mu(x)=\mathbf{I}_3$ or $\mathbf{E}=\mathbf{H}=0$,
the solutions of the system \eqref{eq:max-step11h}-\eqref{eq:max-step1ie} can be written also in the form
\begin{align*}
\int_{\uc{A}} \nabla\wedge\mathbf{E} \cdot \mathbf{v}\, dx
-i \omega\mu_{0}\int_{\uc{A}} \mathbf{H}\cdot\mathbf{v}\, dx & =0,\\
\int_{\uc{A}} \nabla\wedge\mathbf{H}\cdot\mathbf{w}\, dx
+ i \omega\varepsilon_{0}\int_{\uc{A}}\mathbf{E}\cdot\mathbf{w}\, dx& =0,
\end{align*}
which is the weak formulation of
\begin{align*}
\nabla\wedge\mathbf{E}-i\,\omega\mu_{0}\,\mathbf{H} & =0
 \textrm{ in }\uc{A},\\
\nabla\wedge\mathbf{H}+i\,\omega\varepsilon_{0}\,\mathbf{E} & =0
 \textrm{ in }\uc{A}.
\end{align*}
Next, since $A$ is bounded,   
thanks to Rellich's Lemma~\ref{lem:rellich}, $\mathbf{E}=\mathbf{H}=0$
in $\uc{A}\cap \left(\mathbb{R}^3\setminus \bar{B}(R)\right)$, for $R$ large enough.
In particular, $\mathbf{E}$ and $\mathbf{H}$ vanish in a ball contained in $\uc{A}$, which is open and connected, and the conclusion 
follows from Theorem~\ref{thm:unicont}, applied with $\varepsilon(x)=\mu(x)=\mathbf{I}_3$, which in this case reduces to a well known result 
concerning the Helmholtz equation.
\end{proof}
\begin{prop}\label{pro:step-2}
Let
$$
J:=\{i\in I: \left| \mathbf{E}(x) \right|^2 + \left| \mathbf{H}(x) \right|^2 >0 \textrm{ on a set of positive measure in } \Omega_i\}.
$$
Then $J=\emptyset$.
\end{prop}
\begin{proof}
Suppose for contradiction that $J$ is nonempty. Then, by Assumption~\ref{Assumption} there exists $x_0\in \partial\uc{\Omega_J}\cap\partial\Omega_{j_0}$
such that $B(x_0,\delta)\cap \partial\uc{\Omega_J} \subset \partial\Omega_{j_0}$ for some $j_0\in J$ and $\delta>0$. To simplify notation, set $j_0=1$.

Let us show that there exist a point $c$ on $\partial\uc{\Omega_J}\cap\partial\Omega_{1}$ and a radius $\tilde \delta>0$ such that
\begin{equation}\label{eq:step2-1}
 \Omega_J \cap B (c,\tilde\delta)=\Omega_{1}\cap B(c,\tilde\delta).
\end{equation}
Figure~\ref{fig:step2} sketches the configuration we have at hand around $c$.
\begin{figure}\label{fig:step2}
\includegraphics[width=0.3\columnwidth]{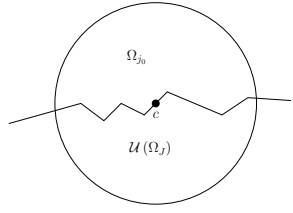}
\caption{A ball centred at an interior point on the boundary of $\uc{\Omega_J}$.}
\end{figure}

Since $\Omega_{1}$ has a $C^{0}$ boundary, for some (smaller) $\delta>0$ there exists a continuous map $f$ and a
suitable orientation of axes such that
$B(x_{0},\delta)\cap\partial\Omega_{J}\subset\partial\Omega_{1}$ and
$$
\Omega_{1}\cap B(x_{0},\delta)
=\left\{ x\in B\left(x_{0},\delta\right): x_{3}>f\left(x_1,x_2\right)\right\}.
$$
This alone does not prove our claim, since $B(x_0,\delta)$ could still intersect $\Omega_J$ when $x_3 \leq f(x_1,x_2)$.
Since $x_{0}\in\partial \uc{\Omega_J}$, there exists a sequence $\left\{y_{j}\right\}\subset\uc{\Omega_J}\cap B(x_{0},\delta)$ such that
$y_{j}$ tends to $x_{0}$.
Consider for a fixed and sufficiently large $j$ the line segment $\{y_{j}+te_{3},t\geq0\}$, and let $\tau>0$ be the least value 
of $t$ such that $y_{j}+t e_{3}\in\partial\Omega_J$.
Then,
$y_{j}+te_{3}\not\in\bar{\Omega}_J,$ for $t<\tau$,
and
$y_{j}+\tau e_{3}\in\partial\Omega_{1}$.
Hence
$y_{j}+\tau e_{3}\not\in\underset{k\in J, k>1}{\cup}\Omega_{k}$.
Since the sets $\Omega_{k}$ are disjoint, the line segment does not intersect
$\underset{k\in J, k>1}{\cup}\Omega_{k}$ in $B(x_{0},\delta)$.
The same argument applies to any line segment $\{z+te_{3},t\geq0\}$ for $z$
sufficiently close to $y_{j}$.
Introducing $c=y_{j}+\tau e_{3}$ we have established that there exists a ball
$B\left(c,\tilde{\delta}\right)$ such that $\Omega_{1}\cap
B\left(c,\tilde{\delta}\right)=\Omega_{J}\cap
B\left(c,\tilde{\delta}\right)$, which is \eqref{eq:step2-1}.

Now, thanks to Proposition~\ref{prop:step1}, and noting that (by Fubini's Theorem) each $\partial\Omega_i$ is of measure zero,
 $\mathbf{E}=\mathbf{H}=0$ almost everywhere in $\uc{\Omega_J}$. 
Thus, for almost every  $x\in B(c,\tilde \delta)$, either $\mathbf{E}=\mathbf{H}=0$ or
$\varepsilon(x) = \varepsilon_1(x)$, and $\mu(x)=\mu_1(x)$. Considering the weak formulation of Maxwell's equations, and arguing as in the proof of
Proposition~\ref{prop:step1}, we note that $\mathbf{E}$ and $\mathbf{H}$ are weak solutions of
\begin{align*}
\nabla\wedge\mathbf{E}-i\,\omega\mu_{0}\mu_{1}(x)\,\mathbf{H} & =0
 \textrm{ in } B(c,\tilde \delta),\\
\nabla\wedge\mathbf{H}+i\,\omega\varepsilon_{0}\varepsilon_{1}(x)\,\mathbf{E} & =0
 \textrm{ in }B(c,\tilde \delta),
\end{align*}
and vanish on the connected non-empty open set $B(c,\tilde \delta)\cap\{x_3<f(x_1,x_2)\}$. 
Since $\epsilon_1$ and $\mu_1$ satisfy \eqref{eq:lip}, that is, 
$$
\| \varepsilon_{1} \|_{W^{1,\infty}(\mathbb{R}^3)^{3\times3}} + \| \mu_{1} \|_{W^{1,\infty}(\mathbb{R}^3)^{3\times3}} \leq M_1,
$$
Theorem~\ref{thm:unicont} shows that $\mathbf{E}=\mathbf{H}=0$ in $B(c,\tilde\delta)$. This in turn shows that $\mathbf{E}$ and $\mathbf{H}$ vanish on a ball inside 
$\Omega_1$, and applying Theorem~\ref{thm:unicont} in $\Omega_1$ we obtain $\mathbf{E}=\mathbf{H}=0$ almost everywhere in $\Omega_1$. This contradiction concludes the proof.
\end{proof}

We now turn to the final step. We have obtained that $\mathbf{E}=\mathbf{H}=0$ almost everywhere in $\Omega$, and therefore either $\mathbf{E}=\mathbf{H}=0$ or $\varepsilon(x)=\mu(x)=\mathbf{I}_3$ 
almost everywhere in $\mathbb{R}^3$. Arguing as above, we deduce that $(\mathbf{E},\mathbf{H})$ is a weak solution of \eqref{eq:maxwell-system} 
with $\varepsilon(x)=\mu(x)=\mathbf{I}_3$ everywhere and the conclusion of Theorem~\ref{thm:no-defect} follows from Rellich's Lemma.

\section{The case of a medium with defects}
We extend our result to the case when defects of small measure
are allowed in the medium. One application is to liquid crystals (see \cite{BASANG-THESIS} for more details). Namely, we assume that the
permittivity and permeability are of the form
\begin{align}
\varepsilon_{D} & = \left(1-\mathbf{1}_{D}\right)\varepsilon
                       +\mathbf{1}_{D}\tilde{\varepsilon},\label{eq:eps-defect}\\
\mu_{D}         & = \left(1-\mathbf{1}_{D}\right)\mu
                       +\mathbf{1}_{D}\tilde{\mu},\nonumber
\end{align}
where $\varepsilon$ and $\mu$  satisfy \eqref{eq:eps-coercive1}-\eqref{eq:localreg-eps}, $\mathbf{1}_{D}$
is the indicator function of a measurable bounded set $D$, such that
\begin{equation}\label{eq:defcon-1}
D\subset{\underset{i\in I}{\cup}}\Omega_{i}, \quad \overline{D\cap\Omega_{i}}\subset\Omega_i \mbox{ and }
\Omega_i \setminus \overline{D\cap\Omega_{i}} \textrm{ is connected for each }i\in I,
\end{equation}
and $\tilde{\epsilon}$ and $\tilde{\mu}$ are real symmetric positive
definite matrices in $L^{\infty}\left(\mathbb{R}^{3}\right)^{3\times3}$ satisfying
\eqref{eq:eps-coercive1}-\eqref{eq:eps-coercive2}.
\begin{thm}\label{thm:defect}
Suppose that the electric and magnetic fields
$\mathbf{E}_{D}\in H_{\textrm{\rm loc}}\left(\textrm{\rm curl};\mathbb{R}^{3}\right)$
and
$\mathbf{H}_{D}\in H_{\textrm{\rm loc}}\left(\textrm{\rm curl};\mathbb{R}^{3}\right)$
are solutions of
\begin{align}
\nabla\wedge\mathbf{E}_{D}-i\,\omega\mu_{0}\mu_{D}(x)\,\mathbf{H}_{D}
 & =  0 \textrm{ in }\mathbb{R}^{3},\label{eq:maxwell-defect}\\
\nabla\wedge\mathbf{H}_{D}+i\omega\varepsilon_{0}\varepsilon_{D}(x)\,\mathbf{E}_{D}
& = 0\textrm{ in }\mathbb{R}^{3},\nonumber
\end{align}
together with the Silver-M\"uller radiation condition \eqref{eq:radiation-condition},
and that $\varepsilon_{D}$ and $\mu_{D}$ are given by \eqref{eq:eps-defect},
with $D$ satisfying \eqref{eq:defcon-1}. Suppose Assumption~\ref{Assumption} holds.
Then, there exists a constant $d_{0}>0$ depending only on the measure $|B_0|$ of $B_0$, $|\omega|$
and the lower and upper bounds $\alpha$ and $\beta$ given in
\eqref{eq:eps-coercive1}-\eqref{eq:eps-coercive2} such that if
the measure of $D$ satisfies $\left|D\right|<d_{0}$, then
$\mathbf{E}_{D}=\mathbf{H}_{D}=0$ almost everywhere.
\end{thm}
To prove Theorem~\ref{thm:defect}, we use the following variant of Theorem~\ref{thm:no-defect}.
\begin{prop}\label{pro:def-1}
Under the same assumptions as Theorem\,\ref{thm:no-defect}, and assuming that \eqref{eq:defcon-1}
holds,
$$
{\rm supp}\,\mathbf{H}_{D}\cup {\rm supp}\,\mathbf{E}_{D}\subset\bar{D}.
$$
\end{prop}
\begin{proof}
The proof follows from that of Theorem\,\ref{thm:no-defect}, since by assumption for each $i\in I$,
$\overline{D\cap\Omega_{i}}\subset\Omega_{i}$,  and the boundary of
$\Omega\setminus\Omega_{i}$ is unaltered by the defects.
\end{proof}
\begin{proof}[Proof of Theorem~\ref{thm:defect}]
Since \eqref{eq:maxwell-defect} admits a weak formulation,
arguing as before we see using Proposition\,\ref{pro:def-1} that
$\mathbf{E}_{D}\in H\left(\textrm{curl};B_{0}\right)$ and
$\mathbf{H}_{D}\in H\left(\textrm{curl};B_{0}\right)$ have
compact support in $B_{0}$ and are also solutions of
\begin{align*}
\nabla\wedge\mathbf{E}_{D}-i\omega\mu_{0}\hat{\mu}\,\mathbf{H}_{D} & =0\textrm{ in }B_{0},\\
\nabla\wedge\mathbf{H}_{D}+i\omega\varepsilon_{0}\hat{\epsilon}\,\mathbf{E}_{D}
& =0\textrm{ in }B_{0},
\end{align*}
where
$\hat{\varepsilon}=\mathbf{I}_3+\mathbf{1}_{D}\left(\tilde{\varepsilon}-\mathbf{I}_3\right),$
and
$\hat{\mu}=\mathbf{I}_3+\mathbf{1}_{D}\left(\tilde{\mu}-\mathbf{I}_3\right)$.
Note that $i\omega\mu_{0}\hat{\mu}\,\mathbf{H}_{D}$ has compact
support and is divergence free. Thus the Helmholtz decomposition (see e.g. \cite{FRIEDRICHS-55,GIRAULT-RAVIART-79,MONK-03}) 
of $i\omega\mu_{0}\hat{\mu}\,\mathbf{H}_{D}$
shows there exists a unique $\mathbf{A}_{H}\in H^{1}\left(B_{0}\right)^{3}$ such that
$\mathbf{A}_{H}\cdot\nu=0,$ on $\partial B_{0}$,
$\textrm{div}\left(\mathbf{A}_{H}\right)=0$ and such that
$i\omega\mu_{0}\hat{\mu}\,\mathbf{H}_{D}=\nabla\wedge\mathbf{A}_{H}$. Furthermore, $\mathbf{A}_{H}$ satisfies
\begin{eqnarray*}
\left\Vert \nabla \mathbf{A}_{H}\right\Vert_{L^{2}\left(B_{0}\right)^{3\times 3}}
&\leq&
C \left( \left\Vert \nabla \wedge \mathbf{A}_{H}\right\Vert_{L^{2}\left(B_{0}\right)^{3}} 
+ \left\Vert \mathbf{A}_{H}\right\Vert_{L^{2}\left(B_{0}\right)^{3}} \right) \\
\mbox{ and }\left\Vert \mathbf{A}_{H}\right\Vert_{L^{2}\left(B_{0}\right)^{3}} 
&\leq&  C |B_0|^{1/3}
\left\Vert \nabla \wedge \mathbf{A}_{H}\right\Vert_{L^{2}\left(B_{0}\right)^{3}} ,
\end{eqnarray*}
where $C$ is a universal constant. Altogether this yields
\begin{equation}
\left\Vert \nabla \mathbf{A}_{H}\right\Vert_{L^{2}\left(B_{0}\right)^{3\times 3}}
\leq C \beta \mu_0 |\omega| \left(\left|B_0\right|+1\right)^{1/3} \left\Vert
\mathbf{H}_{D}\right\Vert _{L^{2}\left(B_{0}\right)^3}. \label{eq:friedrichs}
\end{equation}
Since $\mathbf{E}_{D}-\mathbf{A}_{H}$ is curl free,  we deduce that there exists $p\in H^{1}\left(B_{0}\right)$
such that $\mathbf{E}_{D}=\mathbf{A}_{H}+\nabla p$,  and $p$ is uniquely defined by setting $\int_{B_{0}}p\,dx=0$.
Noticing that $\hat{\varepsilon}\mathbf{E}_{D}$ is divergence free, and $\hat{\varepsilon}-\mathbf{I}_3$ is compactly supported in $B_0$ 
we have that $p$ is the solution of
\begin{align*}
\textrm{div}\left(\hat{\varepsilon}\nabla p\right) & =-\textrm{div}\left(\hat{\varepsilon}\mathbf{A}_{H}\right)\textrm{ in }B_{0},\\
\nabla p\cdot n & =0\textrm{ on }\partial B_{0},\\
\int_{B_{0}}p\,dx & =0.
\end{align*}
Since $\mathbf{A}_H$ is divergence free, the right-hand side becomes
$$
-\textrm{div}\left(\hat{\varepsilon}\mathbf{A}_{H}\right)
=
-\textrm{div}\left(\mathbf{1}_{D}\left(\tilde{\varepsilon}-\mathbf{I}_3\right)\mathbf{A}_{H}\right).
$$
To proceed, we compute using the Cauchy-Schwarz inequality the following bound
\begin{align*}
\alpha\left\Vert \nabla p\right\Vert ^{2} _{L^{2}(B_0)^3} \leq\int_{B_{0}}\hat{\epsilon}\nabla p\cdot\nabla p\,dx
&= -\int_{B_{0}}\mathbf{1}_{D}\left(\tilde{\varepsilon}-\mathbf{I}_3\right)\mathbf{A}_{H}\cdot\nabla p\, dx\\
& \leq \left(\beta+1\right)\left\Vert \mathbf{A}_{H}\right\Vert _{L^{2}(D)^3}
\left\Vert \nabla p\right\Vert _{L^{2}(B_0)^3},
\end{align*}
and we have obtained that
$$
\left\Vert \nabla p\right\Vert _{L^{2}(B_0)^3} \leq \frac{\beta+1}{\alpha}
\left\Vert \mathbf{A}_{H}\right\Vert _{L^2(D)^3}.
$$
Next note using Proposition\,\ref{pro:def-1} that
\[
\left\Vert
\mathbf{E}_{D}\right\Vert _{L^{2}\left(B_{0}\right)^3}
=\left\Vert
\mathbf{E}_{D}\right\Vert _{L^{2}\left(D\right)^3}
\leq
\left\Vert
\nabla p\right\Vert _{L^{2}\left(B_{0}\right)^3}
+\left\Vert
\mathbf{A}_{H}\right\Vert
_{L^{2}\left(D\right)^3}
\leq\frac{2\beta+1}{\alpha}\left\Vert
\mathbf{A}_{H}\right\Vert _{L^{2}\left(D\right)^3}.
\]
The Sobolev-Gagliardo-Nirenberg inequality in $B_{0}$ shows that
$$
\left\Vert \mathbf{A}_{H}\right\Vert _{L^{6}\left(B_{0}\right)^3}\leq
C \left(\left|B_0\right|+1\right)^{1/3}\left\Vert \mathbf{A}_{H}\right\Vert _{H^{1}\left(B_{0}\right)^3},
$$
where $C$ is a universal constant. Therefore, using
H\"older's inequality, together with the Poincar\'e-Friedrichs
estimate \eqref{eq:friedrichs}, we have
\[
\left\Vert \mathbf{A}_{H} \right\Vert_{L^{2}\left(D\right)^3}
\leq
\left|D\right|^{\frac{1}{3}}\left\Vert
\mathbf{A}_{H}\right\Vert _{L^{6}\left(B_{0}\right)^3}\leq
C\beta \left(\left|B_0\right|+1\right)^{2/3} \mu_0 |\omega| \left|D\right|^{\frac{1}{3}}\left\Vert
\mathbf{H}_{D}\right\Vert _{L^{2}\left(B_{0}\right)^3}.
\]
Altogether we have obtained
\begin{equation}
\left\Vert \mathbf{E}_{D}\right\Vert
_{L^{2}\left(B_{0}\right)^3}\leq
C\frac{\beta(\beta+1)}{\alpha}\left(\left|B_0\right|+1\right)^{2/3}  \mu_0 |\omega|\left|D\right|^{\frac{1}{3}}\left\Vert
\mathbf{H}_{D}\right\Vert
_{L^{2}\left(B_{0}\right)^3}.\label{eq:bd-1}
\end{equation}
Repeating the same argument, but starting with $\mathbf{H}_{D}$, we obtain also
\begin{equation}
\left\Vert \mathbf{H}_{D}\right\Vert _{L^{2}\left(B_{0}\right)^3}\leq
C\frac{\beta(\beta+1)}{\alpha}\left(\left|B_0\right|+1\right)^{2/3} \varepsilon_0 |\omega|
\left|D\right|^{\frac{1}{3}}\left\Vert\mathbf{E}_{D}\right\Vert
_{L^{2}\left(B_{0}\right)^3}.\label{eq:bd-2}\end{equation}
The inequalities \eqref{eq:bd-1} and \eqref{eq:bd-2} imply that $\mathbf{H}_{D}=\mathbf{E}_{D}=0$
when
\begin{equation}\label{eq:d0}
\left|D\right|<d_0 := C \frac{\alpha^3}{\beta^3(\beta+1)^3 (\left|B_0\right|+1)^2 \left(\sqrt{\varepsilon_0\mu_0}|\omega|\right)^3}\,,
\end{equation}
where $C$ is a universal constant.
\end{proof}
\begin{rem}
The dependence of the threshold constant $d_0$ given by \eqref{eq:d0} on $|\omega|$ and $\left|B_0\right|$ shows that for a permeability
$\mu$ and a permittivity $\varepsilon$ satisfying \eqref{eq:eps-coercive1}, \eqref{eq:eps-coercive2} and
\eqref{eq:eps-local} only, uniqueness for Maxwell's equations holds provided, if $\omega$ is fixed, the domain $\Omega$ is of small
measure and bounded diameter, or, for a given $\Omega$, when the absolute value of the frequency $|\omega|$ is sufficiently small. 
In such cases, the whole domain 
$\Omega$ can be taken as a defect $D$ (and a fictitious ball containing $D$ plays the role of $\Omega$). We do not claim that 
the dependence of $d_0$ in terms of  $|\omega|$ or $|B_0|$ in \eqref{eq:d0} is optimal. In contrast, Theorem~\ref{thm:no-defect} requires 
additional regularity assumptions  on $\mu$ and $\varepsilon$, but does not depend on the frequency or the size of the domain.
\end{rem}

\section*{Acknowledgements}
The authors were supported by EPSRC Grant EP/E010288/1 and by the  EPSRC  Science and Innovation award to the Oxford Centre for Nonlinear PDE (EP/E035027/1).

\end{document}